\documentclass[12pt,a4paper]{article}
\usepackage{amssymb}
\usepackage{amsmath}
\usepackage{amsthm}
\usepackage{indentfirst}
\usepackage{ot1patch}
\usepackage{psfrag}
\usepackage{rotate}
\usepackage{mathtools}
\usepackage{hyperref}

\newtheorem{tw}{Theorem}[section]
\newtheorem{pr}[tw]{Proposition}
\newtheorem{lm}[tw]{Lemma}

\theoremstyle{definition}

\DeclareMathOperator{\Irr}{Irr}

\author{\L ukasz Matysiak\\
	Kazimierz Wielki University\\
	Bydgoszcz, Poland \\
	lukmat@ukw.edu.pl}
\title{Polynomial composites and certain types of fields extensions}

\begin{document}
	
	\maketitle

\begin{abstract}
In this paper I consider polynomial composites with the coefficients from $K\subset L$. We already know many properties, but we do not know the answer to the question of whether there is a relationship between composites and field extensions. I present the characterization of some known field extensions in terms of polynomial composites.
This paper contains the opening problem of characterization of ideals in polynomial composites with respect to various field extensions.
I also present the full possible characterization of certain field extensions.
\end{abstract}

\begin{table}[b]\footnotesize\hrule\vspace{1mm}
\end{table}

\section{Introduction}
	
By a ring $R$ we mean a commutative ring with unity. 
Denote by $R^{\ast}$ the group of all invertible elements of $R$. The set of all irreducible elements of $R$ be denoted by $\Irr R$. By $\Irr R$ we mean the set of all irreducible elements of $R$.
A Noetherian ring is a ring that satisfies the ascending chain condition on ideals; that is, given any increasing sequence of ideals $I_1\subset I_2\subset\dots $, there exists a natural number (positive integer) $n$ such that $I_n=I_{n+1}=\dots$. There are other, equivalent, definitions: every ideal $I\subset R$ is finitely generated, every non-empty set of ideals of $R$ has a maximal element. 

\medskip

Let $K\subset L$ be a field extension. Let us denote by $[L\colon K]$ the degree of field extension $K\subset L$. In this paper, we will use the following extensions and recall definitions:
\begin{itemize}
	\item[(a) ] a finite extension - the extension that has a finite degree,
	\item[(b) ] an algebraic extension - the extension such that every element of $L$ is algebraic over $K$, i.e. every element of $L$ is a root of some non-zero polynomial with coefficients in $K$,
	\item[(c) ] a separable extension - the algebraic extension such that the minimal polynomial of every element of $L$ over $K$ is separable, i.e., has no repeated roots in an algebraic closure over $K$,
	\item[(d) ] a normal extension - the algebraic extension such that every irreducible polynomial in $K[X]$ that has a root in $L$ completely factors into linear factors over $L$. 
	\item[(e) ] a Galois extension - the algebraic extension such that is both separable and normal.
\end{itemize} 

\medskip

If $K\subset L$ is a Galois extension, then $Aut_K L$ is called the Galois group of $K\subset L$ and denoted by $G(L\mid K)$.
For any subgroup $H$ of $G(L\mid K)$ by $L^H$ denote the corresponding fixed field, i.e. the set of those elements of $L$ which are fixed by every automorphism in $H$.
The algebraic closure of field $K$ denote by $a(K)$.

\medskip

A field $K$ is called perfect if every finite extension of $K$ is separable. For example $\mathbb{Q}$, $\mathbb{C}$, every algebraically closed field.

\medskip

D.D.~Anderson, D.F.~Anderson, M. Zafrullah in \cite{1} called object $A+XB[X]$ as a composite for $A\subset B$ fields. There are a lot of works where composites are used as examples to show some properties. But the most important works are presented below.

\medskip

In 1976 \cite{y1} authors considered the structures in the form $D+M$, where $D$ is a domain and $M$ is a maximal ideal of ring $R$, where $D\subset R$. Later (in \cite{mm1}), I proved that in composite in the form $D+XK[X]$, where $D$ is a domain, $K$ is a field with $D\subset K$, that $XK[X]$ is a maximal ideal of $K[X]$. 
Next, Costa, Mott and Zafrullah (\cite{y2}, 1978) considered composites in the form $D+XD_S[X]$, where $D$ is a domain and $D_S$ is a localization of $D$ relative to the multiplicative subset $S$. 
Zafrullah in \cite{y3} continued research on structure $D+XD_S[X]$ but 
he showed that if $D$ is a GCD-domain, then the behaviour of $D^{(S)}=\{a_0+\sum a_iX^i\mid a_0\in D, a_i\in D_S\}=D+XD_S[X]$ depends upon the relationship between $S$ and the prime ideals $P$ od $D$ such that $D_P$ is a valuation domain (Theorem 1, \cite{y3}).
In 1991 there was an article (\cite{1}) that collected all previous composites and the authors began to create a theory about composites creating results. In this paper, the structures under consideration were officially called as composites. 
After this article, various minor results appeared. But the most important thing is that composites have been used in many theories as examples. 
I have researched many properties of composites in \cite{mm1} and \cite{mm2}.

\medskip

The main motivation of this paper is to answer the following question.

\medskip

\noindent
{\bf Question 1:}\\
Is there a relationship between certain field extensions $K\subset L$ and polynomial composites $K+XL[X]$?

\medskip

In the third chapter I present a full possible characterization of polynomial composites in the form $K+XL[X]$, where $K, L$ are fields, with respect to a given extension with appropriate additional assumptions. I also present a full possible characterization of some extensions of fields $K\subset L$ expressed in the language of polynomial composites $K+XL[X]$ as Noetherian rings with appropriate assumptions.

\medskip

In the fourth chapter I present the consideration of ideals in polynomial composites $K+XL[X]$ assuming that $K\subset L$ is a certain field extension.

\medskip

In the fifth chapter we can found a full posible characterization of considered extensions (Theorems \ref{t3}, \ref{t4}). 

\medskip

This paper is the start of work on the inverse Galois problem. In Galois theory, the inverse Galois problem concerns whether or not every finite group appears as the Galois group of some Galois extension of the rational numbers $\mathbb{Q}$. This problem, first posed in the early 19th century is unsolved. The presented results can be used as mathematical tools. All the Propositions contained in this paper hold for a field of any characteristic, and therefore also for finite fields. We also have a characterization of the Galois extensions (Theorem \ref{t3}, \ref{t4}). The inverse Galois problem can be solved by switching to polynomial composites or to nilpotent elements.

\section{Auxiliary Lemmas}

\begin{lm}
	\label{ll2}
	Let $\varphi\colon K_1\to K_2$ be an isomorphism of fields and $\varPsi\colon K_1[X]\to K_2[X]$ be an isomorphism of polynomials ring. If polynomial $f_1\in \Irr K_1[X]$, $f_1$ has a root $a_1$ in an extended $L_1$ of $K_1$ and polynomial $f_2=\varPsi (f_1)$ has a root $a_2$ in an extended $L_2$ of $K_2$, then there exists $\varPsi'\colon K_1(a_1)\to K_2(a_2)$, which is an extension of $\varphi$ and $\varPsi'(a_1)=a_2$ holds.
\end{lm}

\begin{proof}
	\cite{TC}, Lemma 2, p. 105.
\end{proof}

\begin{lm}
	\label{ll1}
	If $\varphi\colon K\to L$ be an embedding field $K$ to an algebraically closed field $L$, and $K'$ be an algebraic extension of $K$, then there exists an embedding $\varPsi\colon K'\to L$ which is an extension of $\varphi$.
\end{lm}

\begin{proof}
	\cite{TC}, Lemma 4, p. 109.
\end{proof}

\begin{lm}
	\label{ll3}
	If $L$ be a finite field extension of $K$, then $L$ is a Galois extension of $K$ if and only if $|G(L|K)|=(L\colon K)$.
\end{lm}

\begin{proof}
	\cite{TC}, Corollary 1, p. 126.
\end{proof}

\section{Characterization of field extensions in terms of polynomial composites}

Let's start with an auxiliary lemma that will help prove the Theorem \ref{01}.

\begin{lm}
	\label{l1}
	If there exists a nonzero ideal $I$ of $L[X]$, where $L$ is a field, that is finitely generated as an $K+XL[X]$-module, then $K$ is a field and $[L\colon K]<\infty$.
\end{lm}

\begin{proof}
	Clearly, $I$ is finitely generated over $L[X]$, and hence $XL[X]I\neq I$. For otherwise, $XL[X]L[X]_{XL[X]}\cdot IL[X]_{XL[X]}=IL[X]_{XL[X]}$ and therefore $IL[X]_{XL[X]}=0$, by Nakayama's lemma. This is impossible, since $0\neq I\subseteq IL[X]_{XL[X]}$. It follows that $I/XL[X]I$ is a nonzero ($L[X]/XL[X]=L$)-module that is finitely generated as an ($K+XL[X]/XL[X]=K$)-module. Since $L$ is a field, $I/XL[X]I$ can be written as a direct sum of copies of $L$. Thus, $L$ is a finitely generated $K$-module. But then $K$ is a field, since the field $L$ is integral over $K$ and obviously $[L\colon K]<\infty$.
\end{proof}

All my considerations began with the Proposition \ref{01} below. This Proposition motivated me to further consider polynomial composites $K+XL[X]$ in a situation where the extension of fields $K\subset L$ is algebraic, separable, normal and Galois, respectively.

\begin{pr}
	\label{01}
	Let $K\subset L$ be a field extension. Put $T=K+XL[X]$. Then
	$T$ is Noetherian if and only if $[L\colon K]<\infty$.
\end{pr}

\begin{proof}
$\Rightarrow$ 
Since $XL[X]$ is a finitely generated ideal of $K+XL[X]$, it follows from Lemma \ref{l1} that $[L\colon K]<\infty$. Thus, $L[X]$ is module-finite over the Noetherian ring $K+XL[X]$.

\medskip

$\Leftarrow$
$L[X]$ is Noetherian ring and module-finite over the subring $K+XL[X]$. This is the situation covered by P.M. Eakin's Theorem \cite{zzz}.
\end{proof}

\begin{pr}
	\label{02}
	Let $K\subset L$ be a fields extension such that $L^{G(L\mid K)}=K$. Put $T=K+XL[X]$. 
	$T$ is Noetherian if and only if $K\subset L$ be an algebraic extension.
\end{pr}

\begin{proof}
($\Rightarrow$) 
Since $T=K+XL[X]$ is Noetherian, where $K\subset L$ be fields extension, then by Proposition \ref{01} we get that $K\subset L$ is a finite extension. And every finite extension is algebraic. 

\medskip

($\Leftrightarrow$) 
Assume that $K\subset L$ be an algebraic extension. Assuming $L^{G(L\mid K)}=K$ we get directly from the definition of the Galois extension. Since $K\subset L$ be the Galois extension, then $K\subset L$ be a normal extension. Every normal extension is finite, then by Proposition \ref{01} we get that $K+XL[X]$ be a Noetherian.
\end{proof}

\begin{pr}
	\label{04}
	Let $K\subset L$ be fields extension such that $K$ be a perfect field and assume that any $K$-isomorphism $\varphi\colon M\to M$, where $\varphi(L)=L$ holds for every field $M$ such that $L\subset M$.
	 Put $T=K+XL[X]$. 
	$T$ be a Noetherian if and only if $K\subset L$ be a separable extension.
\end{pr}

\begin{proof}
	($\Rightarrow$) 
	By Proposition \ref{01} $K\subset L$ be a finite extension. Every finite extension be an algebraic extension. Since $K$ be the perfect field, then $K\subset L$ be a separable extension.
	
	\medskip
	
	($\Leftrightarrow$)
	First we show that if $L$ be a separable extension of the field $K$, then the smallest normal extension $M$ of the field $K$ containing $L$ be the Galois extension of the field $K$.
	
	\medskip
	
	If $L$ be a separable extension of the field $K$, and $N$ be a normal extension of the field $K$ containing $L$, then let $M$ be the largest separable extension of $K$ contained in $N$. So we have $L\subset M$ and therefore it suffices to prove that $M$ be the normal extension of $K$. 
	
	\medskip
	
	Let $g\in\Irr K[X]$ has a root $a$ in the field $M$. Because $N$ be the normal extension of $K$ and $a\in N$, so it follows that all roots of polynomial $G$ belong to the field $N$. The element $a$ is separable relative to $K$, and so belong to $M$. Hence polynomial $g$ is the product of linear polynomials belonging to $M[X]$, which proves that $M$ be the normal extension of the field $K$.
	
	\medskip
	
	Since $M$ be the normal extension of $K$ and the Galois extension of $K$, then $L$ be the normal extension of $K$ by the assumption (\cite{TC}, exercise 4, p. 119).
	
	\medskip
	
	Because $L$ be the normal extension of $K$, then $L$ be the finite extension of $K$. And by Theorem \ref{01} we get that $K+XL[X]$ is Noetherian.
\end{proof}

\begin{pr}
	\label{06}
	Let $K\subset L$ be fields extension. Assume that if a map $\varphi\colon L\to a(K)$ is $K$-embedding, then $\varphi (L)=L$. 
	Put $T=K+XL[X]$. 
	$T$ be a Noetherian if and only if $K\subset L$ be a normal extension.
\end{pr}

\begin{proof}
	($\Rightarrow$) 
	By Proposition \ref{02} $K\subset L$ be the algebraic extension.
	
	\medskip
	
	Let $c$ be a root of polynomial $g$ belonging to $L$, and $b$ be the arbitrary root of $g$ belonging to $a(K)$. Because polynomial $g\in\Irr K[X]$, so by Corollary from Lemma \ref{ll2} there exists $K$-isomorphism $\varphi'\colon K(c)\to K(d)$. By Lemma \ref{ll1} it can be extended to embedding $\varphi\colon L\to a(K)$. Hence towards $\varphi (L)=L$ and $\varphi (K(c))=\varphi'(K(c))=K(d)$ we get that $K(d)\subset L$, so $b\in L$. Hence every root of polynomial $g$ belong to $L$, so polynomial $g$ is the product of linear polynomials belonging to $L[X]$ ($\ast$). 
	
	\medskip
	
	For every $c\in L$ let $g_c\in\Irr K[X]$ satisfying $g_c(c)=0$. By ($\ast$) every roots of $g_c$ belong to the field $L$. Hence $L$ is a composition of splitting field of all polynomials $g_c$, where $c\in L$. Hence $L$ be the normal extension of $K$.
	
	\medskip
	
	($\Leftrightarrow$)
	 If $L$ be a normal field extension of the field $K$, then $K\subset L$ be the finite extension. Then by Proposition \ref{01} we get that $K+XL[X]$ be Noetherian.
\end{proof}

\begin{pr}
	\label{07}
	Let $K\subset L$ be fields extension such that $L^{G(L\mid K)}=K$. Put $T=K+XL[X]$. 	
	$T$ be a Noetherian if and only if $K\subset L$ be a normal extension.
\end{pr}

\begin{proof}
	($\Rightarrow$)
	By Proposition \ref{02} we get that $K\subset L$ is the algebraic field extension. Assuming $L^{G(L\mid K)}=K$ we get directly from the definition of the Galois extension, and so normal extension. 
	
	\medskip
	
	($\Leftrightarrow$)
	Proposition \ref{07} ($\Leftrightarrow$).
\end{proof}

\begin{pr}
	\label{09}
	Let $T=K+XL[X]$ be Noetherian, where $K\subset L$ be fields. Assume
	$|G(L\mid K)|=[L\colon K]$ and any $K$-isomorphism $\varphi\colon M\to M$, where $\varphi(L)=L$ holds for every field $M$ such that $L\subset M$.
	$T$ be a Noetherian if and only if $K\subset L$ be a Galois extension. 
\end{pr}	

\begin{proof}
	($\Rightarrow$)
	By Proposition $\ref{01}$ we get $K\subset L$ be the finite extension. By the assumption we can use Lemma \ref{ll3} and we get that $K\subset L$ be Galois extension.
	
	($\Leftarrow$)
	If $K\subset L$ be Galois fields extension, then is separable. By Proposition \ref{04} we get that $K+XL[X]$ be Noetherian.
\end{proof}

\begin{pr}
	\label{10}
	Let $T=K+XL[X]$, where $K\subset L$ be fields such that $K=L^{G(L\mid K)}$. $T$ be a Noetherian if and only if $K\subset L$ be a Galois extension. 
\end{pr}	

\begin{proof}
	($\Rightarrow$) 
	By Proposition $\ref{02}$ we get that $K\subset L$ be the algebraic extension. Assuming $K=L^{G(L\mid K)}$ we get directly from the definition of the Galois extension.
	
	\medskip
	
	($\Leftarrow$)
	If $K\subset L$ be Galois fields extension. Then $K\subset L$ be a normal extension. Hence by Proposition \ref{07} we get that $K+XL[X]$ be Noetherian.
\end{proof}

\begin{pr}
	\label{13}
	Let $K\subset L\subset M$ be fields such that $K$ be a perfect field. If $K+XL[X]$ and $L+XM[X]$ be Noetherian then $K\subset M$ be separable fields extension.
	
	\medskip
	
	Moreover, if we assume that any $K$-isomorphism $\varphi\colon M'\to M'$, where $\varphi(M)=M$ holds for every field $M'$ such that $M\subset M'$, then $K+XM[X]$ be a Noetherian. 
\end{pr}
	
\begin{proof}
	By Proposition \ref{04} we get $K\subset L$, $L\subset M$ be separable extensions. Then $K\subset M$ be separable extension. Moreover, from \ref{04} we get $K+XM[X]$ is Noetherian.
\end{proof}

\begin{pr}
	\label{14}
	Let $K\subset L\subset M$ be fields such that $M^{G(M\mid K)}=K$. If $K+XM[X]$ be Noetherian then $L\subset M$ be a normal fields extension. Moreover, $L+XM[X]$ be Noetherian.
\end{pr}

\begin{proof}
	By Proposition \ref{07} we have that $K\subset M$ be a normal extension. Then $L\subset M$ be the normal extension. Moreover, from \ref{04} we get that $L+XM[X]$ be Noetherian.  
\end{proof}

\begin{pr}
	\label{15}
	Let $K\subset L$ be extension fields such that $[L\colon K]=2$. Then $K+XL[X]$ be Noetherian. Moreover, if $L^{G(L\mid K)}=K$, then $K\subset L$ be a normal.
\end{pr}

\begin{proof}
	Of course, from Propositon \ref{01} we get $K+XL[X]$ is Noetherian. By Proposition \ref{07} we have $K\subset L$ be normal fields extension.  
\end{proof}

\section{Fields extension and ideals in composites}

In this chapter, let's consider how can ideals be characterized in such polynomial composites, assuming that a given fields extension is a certain type?	
	
\begin{pr}
	\label{17}
	Let $K\subset L$ be fields and algebraic extension such that $K=L^{G(L\mid K)}$. Then every ideal of $K+XL[X]$ is finite generated.
\end{pr}

\begin{proof}
	By Propositon \ref{02} we have that $K+XL[X]$ is Noetherian. Hence every ideal of $K+XL[X]$ is finite generated.
\end{proof}

\begin{pr}
	\label{18}
	Let $K\subset L$ be fields and finite extension. Then every ideal of $K+XL[X]$ is finite generated.
\end{pr}

\begin{proof}
	By Propositon \ref{01} we have that $K+XL[X]$ is Noetherian. Hence every ideal of $K+XL[X]$ is finite generated.
\end{proof}
	
Unfortunately, the following questions have arisen at the moment.
	
\medskip
	
\noindent
{\bf Questions}
\label{19}
	
\begin{itemize}
	\item[1. ]
	What is the additional argument of the Proposition \ref{17} if we assume that fields extension be separable?
		
	\item[2. ]
	What is the additional argument of the Propostion \ref{17} if we assume that fields extension be normal?
\end{itemize}

\section{Full characterization}
\label{R5}

In this section, we present the full possible characterization of field extensions. Combining the Magid results and from this paper, we get the following two theorems.

\begin{tw}[\cite{Magid}, Theorem 1.2.]
	\label{t1}
	Let $M$ be an algebraically closed field algebraic over $K$, and let $L$ such that $K\subseteq L\subseteq M$ be an intermediate field. Then the following are equivalent:
	\begin{itemize}
		\item[(a) ] $L$ is separable over $K$.
		\item[(b) ] $M\otimes_K L$ has no nonzero nilpotent elements.
		\item[(c) ] Every element of $M\otimes_K L$ is a unit times an idempotent.
		\item[(d) ] As an $M$-algebra $M\otimes_KL$ is generated by idempotents.
	\end{itemize}
\end{tw}

\begin{tw}[\cite{Magid}, Theorem 1.3.]
	\label{t2}
	Let $M$ be an algebraically closed field containing $K$, and let $L$ be a field algebraic over $K$. Then the following are equivalent:
	\begin{itemize}
		\item[(a) ] $L$ is separable over $K$.
		\item[(b) ] $M\otimes_K L$ has no nonzero nilpotent elements.
		\item[(c) ] Every element of $M\otimes_K L$ is a unit times an idempotent.
		\item[(d) ] As an $M$-algebra $M\otimes_KL$ is generated by idempotents.
	\end{itemize}
\end{tw}	

Below we have conclusions from the above results.

\begin{tw}
	\label{t3}
	In Theorems \ref{t1} and \ref{t2} if assume $L^{G(L\mid K)}=K$, then conditions (a) -- (d) are equivalent to
	\begin{itemize}
		\item[(e) ] $K+XL[X]$ be a Noetherian.
		\item[(f) ] $[L\colon K]<\infty$
		\item[(g) ] $K\subset L$ be an algebraic extension.
		\item[(h) ] $K\subset L$ be a Galois extension.
	\end{itemize}
\end{tw}

\begin{proof}
	(h)$\Rightarrow$(a) -- Obvious.
	
	\medskip
	
	(a)$\Rightarrow$(g)$\Rightarrow$(e)$\Rightarrow$(h) If $K\subset L$ be a separable extension, then be an algebraic extension. By Proposition \ref{02} $K+XL[X]$ be a Noetherian. By Proposition \ref{10} $K\subset L$ be a Galois extension.
	
	\medskip
	
	(e)$\Rightarrow$(f) -- Proposition \ref{01}.
\end{proof}

\begin{tw}
	\label{t4}
	In Theorem \ref{t3} if assume $K$ be a perfect field and $L^{G(L\mid K)}=K$, then conditions (a) -- (h) are equivalent to \\ 
	(g) $K\subset L$ be a normal extension.
\end{tw}

\begin{proof}
	(g)$\Rightarrow$(a) If $K\subset L$ be a normal extension, then be an algebraic extension. By definition perfect field $K\subset L$ be a separable extension.
	
	\medskip
	
	(h)$\Rightarrow$(g) Obvious.
\end{proof}

Proposition \ref{10}, Theorems \ref{t3} and \ref{t4} can be used to solve the inverse Galois problem. There is a lot of work. And it is enough to solve the problem for nonabelian groups. Thus, the following question arises:

\medskip

\noindent
{\bf Question 2:}\\
Can all the statements in this paper operate in noncommutative structures?

\medskip

And another question also arises regarding polynomial composites:

\medskip

\noindent
{\bf Question 3:}\\
Under certain assumptions for any type of $K\subset L$, we get that $K+XL[X]$ be a Noetherian ring. When can $K+XL[X]$ be isomorphic to any Noetherian ring?

\end{document}